\documentclass[11pt]{article}
\usepackage{mystyle}

\begin{document}
\title{The independence ratio of 4-cycle-free planar graphs}
\author{Tom Kelly}
\author{Sid Kolichala}
\author{Caleb McFarland}
\author{Jatong Su}
\affil[]{Georgia Institute of Technology, Atlanta, Georgia}
\affil[]{{\texttt{\{tom.kelly, skolichala3, cmcfarland30, jsu77\}@gatech.edu}}}
\date{}
\maketitle

\begin{abstract}
    We prove that every $n$-vertex planar graph $G$ with no triangle sharing an edge with a 4-cycle has independence ratio $n/\alpha(G) \leq 4 - \varepsilon$ for $\varepsilon = 1/30$. This result implies that the same bound holds for 4-cycle-free planar graphs and planar graphs with no adjacent triangles and no triangle sharing an edge with a 5-cycle. For the latter case we strengthen the bound to $\varepsilon = 2/9$.
\end{abstract}

\section{Introduction}\label{introduction-section}

In 1976, Appel and Haken \cite{appel1976every} proved the Four Color Theorem. That same year, Garey, Johnson, and Stockmeyer 
\cite{Garey1976ThreeColNP} showed that determining whether a planar graph is 3-colorable is NP-complete. Since then, sufficient conditions under which planar graphs are 3-colorable have been a topic of much research, specifically for planar graphs in which cycles of particular lengths are forbidden. Note that either 3-cycles or 4-cycles must be forbidden because of $K_4$, which has chromatic number four. It had already been proven by Gr\"{o}tzsch \cite{grotzsch1959zur} in 1959 that triangle-free planar graphs are 3-colorable, so attention was devoted both to strengthening Gr\"{o}tzsch's result and planar graphs that excluded 4-cycles.

Steinberg \cite{steinberg1993stateOf3CP} famously conjectured in 1976 that planar graphs without 4- or 5-cycles are 3-colorable, but Steinberg's Conjecture was disproved by Cohen-Addad, Hebdige, Kr\'{a}l, Li, and Salgado \cite{cohen2017steinberg} in 2017. Prior to that, Erd\H{o}s proposed the following relaxation of this problem: does there exists a constant $C$ such that every planar graph without cycles of length four to $C$ is 3-colorable? In 2005, Borodin, Glebov, Raspaud, and Salavatipour \cite{borodin2005planar} proved that planar graphs without cycles from length four to seven are 3-colorable. It still remains open whether every planar graph without cycles of length four, five, and six is 3-colorable. However, many planar graphs with 4-cycles and other excluded configurations were proven to be 3-colorable \cite{ borodin2003SufficientCodition, borodin2011planar, borodin2010planar, Choi2018precolored9cycle, Dvorak2015precolored8cycle,  La2022extensions, liang2023note, lu2009without479cycles, xu20063}.

Another way to relax Steinberg's Conjecture is to consider so-called \textit{improper colorings}. A graph is \textit{$(c_1, c_2, ..., c_k)$-colorable} if its vertex set can be partitioned into $k$ sets $V_1, V_2, ..., V_k$ such that for every $i \in \{1, ..., k\}$, the maximum degree of the induced subgraph $G[V_i]$ is at most $c_i$. In 2013, Hill and Yu \cite{Hill2013relaxSteinberg} and independently Xu, Miao, and Wang \cite{xu2014every(110)} proved that planar graphs without 4- or 5-cycles are $(1, 1, 0)$-colorable, and Chen, Wang, Liu, and Xu \cite{chen2016planar(200)} proved they are $(2, 0, 0)$-colorable. Others have proved similar results about improperly coloring planar graphs \cite{huang2020every, Huang2022relaxNovosibirsk, huang2019decomposing, Kang2022(100)color, Li2021(200)col, liu2016planar, liu2023(33)-col, miao2018planar, zhang2016planar}. It remains open whether planar graphs without 4- or 5-cycles are $(1,0,0)$-colorable.


The \textit{fractional chromatic number} $\chi_f(G)$ of a graph $G$ is the infimum of all $a/b$ such that to every vertex of $G$ one can assign a subset of $\{1, 2, ..., a\}$ of size $b$ in such a way that adjacent vertices are assigned disjoint sets. In 1973, before the Four Color Theorem was proven, Hilton, Rado, and Scott \cite{hilton1973a5colour} proved that every planar graph has fractional chromatic number less than five. In 2015 Dvo\v{r}\'{a}k, Sereni, and Volec \cite{dvorak2015fractional} showed that every $n$-vertex triangle-free planar graph has fractional chromatic number at most $3 - 3/(3n + 1)$, which is best possible. Dvo\v{r}\'{a}k and Mnich \cite{dvorak2015fractional} conjectured that every $n$-vertex planar graph of girth at least five has fractional chromatic number at most $3 - \varepsilon$ for some $\varepsilon > 0$. Relaxing Steinberg's Conjecture yet another way, Dvo\v{r}\'{a}k and Hu \cite{dvovrak2019planar} proved in 2019 that planar graphs without 4- or 5-cycles have fractional chromatic number at most $11/3$. Other similar contributions to the fractional chromatic number of planar graphs have been made \cite{cranston2018planar, dvorak2008planar, dvorak2020fractional, pirnazar2002girth, wu2020planar}.

The \textit{independence number} of a graph $G$, denoted $\alpha(G)$, is the size of a largest independent set of $G$. The \textit{independence ratio} of a graph $G$ is $|V(G)|/\alpha(G)$. It is well known that
\begin{equation}\label{independence-fractional-chromatic}
    \frac{|V(G)|}{\alpha(G)} \leq \chi_f(G) \leq \chi(G),
\end{equation}
where $\chi(G)$ is the chromatic number of $G$. Just as it is natural to find the supremum of the chromatic numbers of a class of graphs, it is natural to find the supremum of their independence ratios. If a connected graph $G$ is not a complete graph or an odd cycle, its independence ratio is already bounded from above by the maximum degree of $G$ due to Brooks' Theorem \cite{brooks1941colouring}. Staton \cite{staton1979some} first proved in 1979 that triangle-free graphs with maximum degree three have independence ratio at most $14/5$. Jones \cite{jones1984independence} then proved in 1984 that the supremum of the independence ratio of triangle-free graphs with maximum degree four is $13/4$. Gr\"{o}tzsch's Theorem implies that the independence ratio of triangle-free planar graphs is at most three. Steinberg and Tovey \cite{steinberg1993ramsey} showed in 1993 that every $n$-vertex triangle-free planar graph has an independent set of size $\lfloor n/3 \rfloor + 1$. In 2017, Dvo\v{r}\'{a}k and Mnich \cite{dvorak2017large} improved on that result by proving that $(n + 1)/3$ is a tight lower bound on the independence number of such planar graphs.

We too turn our attention to the independence ratio, but in the context of Steinberg's Conjecture. Note that the result of Dvo\v{r}\'{a}k and Hu \cite{dvovrak2019planar} together with \eqref{independence-fractional-chromatic} implies planar graphs with no 4- or 5-cycles have independence ratio at most $11/3$. However, since Steinberg's Conjecture is false, it is natural to consider forbidding only 4-cycles. So, we pose the following problem.

\begin{problem}\label{sup-4-cycle-free-IR-prob}
    Determine the supremum $\mathscr{A}$ of the independence ratio of 4-cycle-free planar graphs.
\end{problem}

We know that $\mathscr{A} \ge 16/5$ from Figure \ref{fig:probEpsBound}, which displays a 16-vertex graph with independence number five. 
\begin{figure}[h]
    \centering
    \includegraphics[]{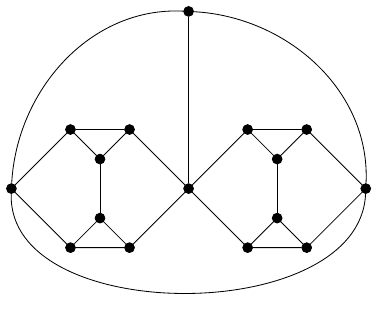}
    \caption{A 4-cycle-free planar graph with independence ratio $16/5$}
    \label{fig:probEpsBound}
\end{figure}

By the Four Color Theorem, $\mathscr{A} \le 4$. To show that $\mathscr{A} < 4$, we prove the following theorem.

\begin{thm}\label{4-cycle-free-IR-thm}
    Every 4-cycle-free planar graph has independence ratio at most $4 - 1/30$.
\end{thm}

It would also be interesting to determine the supremum of the fractional chromatic number of 4-cycle-free planar graphs. To that end, we conjecture the following.

\begin{conj}\label{frac-chromatic-conjecture}
    There exists some $\varepsilon > 0$ such that the fractional chromatic number of every 4-cycle-free planar graph is at most $4 - \varepsilon$.
\end{conj}

We say two cycles are \textit{adjacent} if they have a common edge. In 2003, before Steinberg's Conjecture was disproved, Borodin and Raspaud \cite{borodin2003SufficientCodition} posed the Strong Bordeaux Conjecture stating that every planar graph without adjacent 3-cycles and without 5-cycles is 3-colorable, which would imply Steinberg's Conjecture. Borodin, Glebov, Jensen and Raspaud \cite{borodin2010planar} then posed the Novosibirsk Conjecture in 2006, which states that every planar graph without 3-cycles adjacent to cycles of length three or five is 3-colorable. A proof of the Novosibirsk Conjecture would imply the Strong Bordeaux Conjecture and Gr\"{o}tzsch's Theorem. However, since Steinberg's Conjecture is false, the Strong Bordeaux and Novosibirsk Conjecture are false as well. Nevertheless, it is natural to ask the analogue of Problem \ref{sup-4-cycle-free-IR-prob} for such graphs. We show that planar graphs without 3-cycles adjacent to cycles of length three or five have independence ratio less than four by proving the following theorem.

\begin{thm}\label{Novosibirsk-IR-thm}
    Every planar graph without a triangle adjacent to a cycle of length three or five has independence ratio at most $4 - 2/9$.
\end{thm}

It would also be interesting to prove Conjecture \ref{frac-chromatic-conjecture} for planar graphs with no triangle adjacent to a cycle of length three or five. 

We derive both Theorem \ref{4-cycle-free-IR-thm} and Theorem \ref{Novosibirsk-IR-thm} from a more general result. To that end, we introduce the following terminology. We say a component of a graph $G$ is a \textit{2-chain of triangles} if it is isomorphic to the graph formed by adding an edge between two disjoint triangles. A component $C$ of $G$ is said to be \textit{difficult} if $C$ is a triangle or a 2-chain of triangles. We define $\lambda(G)$ to be the number of difficult components of $G$. We prove the following in Section \ref{technical-thm-section}.

\begin{thm}\label{technical-thm}
    If $G$ is a graph without a triangle adjacent to a 4-cycle and $a, b$ are positive constants satisfying inequalities \eqref{deg-1-equation}--\eqref{3-reg-equation}, then $\alpha(G) \geq |V(G)| - (a|V(G)| + b|E(G)| + b\lambda(G))$.
\end{thm}
\begin{align}
    2a + b &\geq 1 \label{deg-1-equation}\\
    3a + 4b &\geq 2 \label{deg-2-equation}\\
    9a + 11b &\geq 6 \label{3-chain-equation}\\
    a + 5b &\geq 1 \label{deg-5-equation}\\
    4a + 9b &\geq 3 \label{deg-4-deg-3-nbr-equation}\\
    10a + 16b &\geq 7 \label{deg-4-deg-3-nbr-difficult-equation}\\
    13a + 20b &\geq 8 \label{deg-4-deg-3-nbr-2-difficult-equation}\\
    5a + 14b &\geq 4 \label{deg-4-deg-4-nbr-equation}\\
    a + \frac{3}{2}b &\geq \frac{2}{3} \label{3-reg-equation}
\end{align}

In fact, Theorem \ref{technical-thm} holds with \eqref{deg-2-equation} and \eqref{deg-4-deg-3-nbr-equation}--\eqref{3-reg-equation} omitted. See Section \ref{conclusion-section} for more details. Note that in Theorem \ref{technical-thm}, we do not require $G$ to be planar. In Section \ref{density-4-cycle-free-section} we show that the supremum of the average degree of planar graphs without a triangle adjacent to a 4-cycle is $30/7$. We use this result along with Theorem \ref{technical-thm} to prove the following.

\begin{thm}\label{triangle-bordering-4-cycle-IR-thm}
    Every planar graph without a triangle adjacent to a 4-cycle has independence ratio at most $4 - 1/30$.
\end{thm}

Theorem \ref{triangle-bordering-4-cycle-IR-thm} immediately implies Theorem \ref{4-cycle-free-IR-thm}. Note that a graph with a triangle adjacent to a 4-cycle contains either adjacent triangles or a triangle adjacent to a 5-cycle, so Theorem \ref{triangle-bordering-4-cycle-IR-thm} also implies a weaker version of Theorem \ref{Novosibirsk-IR-thm}. In Section \ref{density-novosibirsk-section}, we show that the supremum of the average degree of planar graphs without 3-cycles adjacent to cycles of length three or five is four. We use this result along with Theorem \ref{technical-thm} to prove Theorem \ref{Novosibirsk-IR-thm}.

\section{Preliminaries}
In this section we provide a brief overview of notation that will be used throughout the paper. All graphs are simple and undirected. A graph is said to be \textit{$i$-cycle-free} if it does not contain a cycle on $i$ vertices as a subgraph (not necessarily induced). We use $V(G)$ and $E(G)$ to denote the set of vertices and edges of a graph $G$ respectively. If a graph $G$ is embedded in the plane, we use $F(G)$ to denote its faces and $\ell(f)$ to denote the length of a face $f \in F(G)$. For $X \subseteq V(G)$, we denote by $G[X]$ the graph $G$ induced on $X$. We denote by $G \setminus X$ the graph obtained from $G$ by deleting the vertices in $X$. For $v \in V(G)$, we use $d_G(v)$ and $N_G(v)$ to denote the degree and neighborhood of $v$ respectively. We may drop the subscript if the graph $G$ is clear from context. For $H$ a subgraph of a graph $G$, we use $\Phi(H)$ to denote the set of edges with one end in $V(H)$ and one end in $V(G) \setminus V(H)$. We use $\alpha(G)$ to denote the size of a largest independent set in $G$.

\section{Proof of Theorem \ref{technical-thm}}\label{technical-thm-section}
Suppose towards a contradiction that Theorem \ref{technical-thm} is false, and let $G$ be a counterexample with $|V(G)|$ minimum. We proceed with a series of claims. Our first claim is the following inductive lemma which is analogous to \cite[Claim 1]{HeckmanThomas01}.

\begin{claim}\label{main-inequality}
    Let $X \subseteq V(G)$ be nonempty, and let $G' = G \setminus X$. If every independent set $I'$ in $G'$ can be extended to an independent set in $G$ of size at least $|I'| + A$ for some non-negative integer $A$, then $$A + (a - 1)N + bM < b\Lambda,$$ where $N = |X|, M = |E(G)| - |E(G')|, \text{and } \Lambda = \lambda(G') - \lambda(G)$.  
\end{claim}

\begin{proof}
    The graph $G'$ satisfies the conclusion of Theorem \ref{technical-thm} by the minimality of $|V(G)|$. Therefore, 
    \begin{align*}
        \alpha(G) &< |V(G)| - (a|V(G)| + b|E(G)| + b\lambda(G)) \\
        \intertext{\parbox{\linewidth}{\centering and}}
        \alpha(G') &\geq |V(G')| - (a|V(G')| + b|E(G')| + b\lambda(G'))
    \end{align*}
    Suppose that every independent set $I'$ in $G'$ can be extended to an independent set in $G$ of size at least $|I'| + A$ for some non-negative integer $A$. That means $\alpha(G) - \alpha(G') \ge A$, so by combining this with the above inequalities, we get
    \begin{equation*}
        A \le \alpha(G) - \alpha(G') < N - aN - bM + b\Lambda,
    \end{equation*}
    and the result follows by rearranging terms.
\end{proof}

Our remaining claims show the existence of a set of reducible configurations, with at least one configuration appearing in $G$. In these claims, we specify the set $X$ for which we apply Claim \ref{main-inequality}. The choice of $X$ determines $A$, $N$, $M$, and $\Lambda$.

\begin{claim}\label{min-degree-1}
    The minimum degree of $G$ is at least one.
\end{claim}
\begin{proof}
    Suppose towards a contradiction that $G$ contains an isolated vertex $v$. Let $X = \{v\}$, so $A = N = 1$ and $M = \Lambda = 0$. Claim \ref{main-inequality} implies
    \begin{equation*}
        1 + (a - 1) < 0.
    \end{equation*}
    However, this contradicts that $a$ is positive.
\end{proof}

\begin{claim}\label{min-degree-2}
    The minimum degree of $G$ is at least two.
\end{claim}
\begin{proof}
    Suppose towards a contradiction that $G$ contains a vertex $v$ of degree at most one. By Claim \ref{min-degree-1}, $d_G(v) = 1$. Let $X$ consist of $v$ and its neighbor $u$. Note that $A = 1$, $N = 2$, $M = d_G(u)$, and $\Lambda \leq d_G(u) - 1$. Thus, Claim \ref{main-inequality} implies
    \begin{equation*}
        1 + 2(a - 1) + bM < b(M - 1).
    \end{equation*}
    However, this contradicts \eqref{deg-1-equation}.
\end{proof}

\begin{claim}\label{lambda-equals-zero}
    $\lambda(G) = 0.$
\end{claim}
\begin{proof}
    Suppose towards a contradiction that $\lambda(G) \neq 0$. That means at least one of the components of $G$ is a triangle or 2-chain of triangles. If $G$ contains a triangle component, let $X$ consist of the vertices of the triangle component. In this case, $A = 1, N = 3, M = 3$, and $\Lambda = -1$. Thus, Claim \ref{main-inequality} implies
    \begin{equation*}
        1 + 3(a - 1) + 3b < -b.
    \end{equation*}
    However, this contradicts \eqref{deg-2-equation}. If $G$ contains a 2-chain of triangles, let $X$ consist of the vertices of the 2-chain of triangles. In this case, $A = 2, N = 6, M = 7$, and $\Lambda = -1$. Thus, Claim \ref{main-inequality} implies
    \begin{equation*}
        2 + 6(a - 1) + 7b < -b.
    \end{equation*}
    However, this also contradicts \eqref{deg-2-equation}.
\end{proof}

If $H$ is an induced subgraph of $G$, we let $\Phi(H)$ denote the number of edges of $G$ with one end in $V(H)$ and one end in $V(G) \setminus V(H)$.

\begin{claim}\label{no-deg-2-in-triangle}
    $G$ contains no degree-2 vertices that are part of a triangle.
\end{claim}
\begin{proof}
    Suppose towards a contradiction that $G$ contains a degree-2 vertex that is a part of a triangle, $J$, and choose $J$ such that $\Phi(J)$ is minimum. By Claim \ref{lambda-equals-zero}, $\Phi(J) \ge 1$. Let $X$ consist of the vertices of $J$, so $A = 1, N = 3$, and $M = 3 + \Phi(J)$. Claim \ref{main-inequality} implies
    \begin{equation*}
        1 + 3(a - 1) + (3 + \Phi(J))b < \Lambda b.
    \end{equation*}
    By \eqref{deg-2-equation}, $\Lambda \geq \Phi(J)$. By the choice of $J$ with $\Phi(J)$ minimum, either $\Phi(J) = 1$ or every other triangle $H$ satisfies $\Phi(H) \geq 2$, so $\Lambda \leq \max\{1, \Phi(J)/2\}$. Thus, $\Lambda = \Phi(J) = 1$. Claim \ref{lambda-equals-zero} implies $G \setminus X$ does not contain a triangle component, as otherwise $G$ contains a 2-chain of triangles. Therefore, $G \setminus X$ contains a 2-chain of triangles. Let $X'$ consist of $X$ and the vertices of this 2-chain of triangles and note $\Phi(G[X']) = 0$. Hence, $\lambda(G \setminus X') = 0$, so Claim \ref{main-inequality} with $X'$ playing the role of $X$ implies
    \begin{equation*}
        3 + 9(a - 1) + 11b < 0.
    \end{equation*}
    However, this contradicts \eqref{3-chain-equation}.
\end{proof}

\begin{claim}\label{triangles-edges-leaving}
    Every triangle $H$ in $G$ satisfies $\Phi(H) \geq 3$.
\end{claim}
\begin{proof}
    By Claim \ref{no-deg-2-in-triangle}, all three vertices in $V(H)$ must have a neighbor in $V(G) \setminus V(H)$.
\end{proof}

\begin{claim}\label{2-chain-edges-leaving}
    If $H$ induces a 2-chain of triangles in $G$, then $\Phi(H) \geq 4$.
\end{claim}
\begin{proof}
    By Claim \ref{no-deg-2-in-triangle}, all four vertices with degree two in $H$ must have a neighbor in $V(G) \setminus V(H)$.
\end{proof}

We use Claims \ref{triangles-edges-leaving} and \ref{2-chain-edges-leaving} to show that $\lambda(G\setminus{X}) =0$ for various choices of $X$.

\begin{claim}\label{min-deg-3}
    The minimum degree of $G$ is at least three.
\end{claim}
\begin{proof}
    By Claim \ref{min-degree-2}, the minimum degree of $G$ is at least two. Suppose towards a contradiction that there exists a degree-2 vertex $v$ in $G$. Let $X$ consist of $v$ and its two neighbors, $u$ and $w$, so $A = 1, N = 3$, and $M = d_G(u) + d_G(w)$ by Claim \ref{no-deg-2-in-triangle}. By Claim \ref{min-degree-2}, the degrees of $u$ and $w$ are each at least two. Note $G \setminus X$ cannot contain a triangle component because one of $u$ or $w$ would have to be adjacent to two vertices of the triangle by Claim \ref{triangles-edges-leaving}, so $G$ would contain triangles sharing an edge. Thus, by Claim \ref{2-chain-edges-leaving}, $\Lambda \le \lfloor (d_G(u) + d_G(w) - 2)/4 \rfloor$. Claim \ref{main-inequality} implies
    \begin{equation*}
        1 + 3(a - 1) + (d_G(u) + d_G(w))b < b \Biggl\lfloor \frac{d_G(u) + d_G(w) - 2}{4} \Biggr\rfloor.
    \end{equation*}
    However, this contradicts \eqref{deg-2-equation}. Indeed, if $d_G(u) + d_G(v) \ge 5$, the floor function can be ignored and we can still achieve a contradiction, and for $d_G(u) + d_G(v) = 4$, we still obtain a contradiction.
\end{proof}

\begin{claim}\label{max-deg-4}
    The maximum degree of $G$ is at most four.
\end{claim}
\begin{proof}
    Suppose towards a contradiction that there exists a vertex $v$ with degree greater than four. Let $X = \{v\}$, so $A = 0, N = 1$, and $M = d_G(v)$. Note that $G \setminus X$ cannot contain a triangle component because $v$ would have to be adjacent to three vertices of the triangle by Claim \ref{triangles-edges-leaving}, so $G$ would contain triangles sharing an edge. Likewise by Claim \ref{2-chain-edges-leaving}, $G \setminus X$ cannot contain a 2-chain of triangles or else $G$ would contain triangles sharing an edge. Thus, $\Lambda = 0$, and Claim \ref{main-inequality} implies
    \begin{equation*}
        a - 1 + d_G(v)b < 0.
    \end{equation*}
    However, this contradicts \eqref{deg-5-equation}.
\end{proof}

\begin{claim}\label{deg-4-no-deg-3-nbrs}
    Every degree-3 vertex in $G$ has no degree-4 neighbors.
\end{claim}
\begin{proof}
    Suppose towards a contradiction that $G$ contains a degree-3 vertex $v$ with a degree-4 neighbor $u$. Let $X$ consist of $v$ and its neighbors, so $A = 1$ and $N = 4$. Since $G$ does not contain a triangle adjacent to a 4-cycle, $G[N(v)]$ has at most one edge. Thus, by Claims \ref{min-deg-3} and \ref{max-deg-4}, we have $9 \leq M \leq 12$. Claim \ref{main-inequality} implies
    \begin{equation*}
        1 + 4(a - 1) + Mb < \Lambda b.
    \end{equation*}
    By \eqref{deg-4-deg-3-nbr-equation}, $M - \Lambda \leq 8$. However, Claims \ref{triangles-edges-leaving} and \ref{2-chain-edges-leaving} imply $\Lambda \leq \lfloor \Phi(G[X])/3 \rfloor \leq \lfloor (M - 3)/3 \rfloor$, so $M \in \{9, 10\}$. Moreover, if $M = 9$, then $\Phi(G[X]) = 5$ and $\Lambda = 1$, and if $M = 10$, then $\Lambda = 2$. Note that in either case, if neighbors of $v$ are adjacent, then $G \setminus X$ cannot contain a triangle component. Suppose it did, so by Claim \ref{triangles-edges-leaving} and because $G$ has no adjacent triangles, the vertices of a triangle component must be adjacent to every neighbor of $v$. However, because two neighbors of $v$ are adjacent, they would belong to a 4-cycle that shares an edge with a triangle in $G$.

    \begin{figure}[H]
        \centering
        \subfloat[\centering $G \setminus X$ contains two triangle components]{



              
            \includegraphics[]{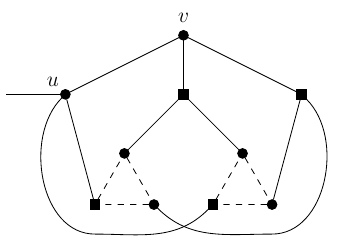}\label{subfig:M=10-tri-tri}
        }
        \qquad
        \subfloat[\centering $G \setminus X$ contains a triangle component and a 2-chain of triangles]{



              
            \includegraphics[]{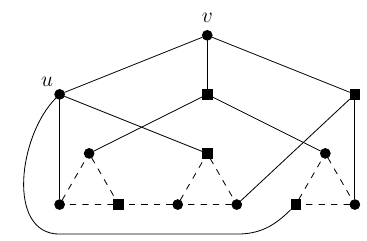}
            \label{subfig:M=10-tri-2chain}
        }
        
        \caption{The two configurations in Case 1. Edges of difficult components are depicted with dashed lines, and any independent set in $G \setminus X'$ can be extended by adding the square vertices.}
        \label{fig:deg-3-deg-4-nbr-M=10}
    \end{figure}

    \textbf{Case 1.} $M = 10$ and $\Lambda = 2$. 
    
    In this case, $G[N(v)]$ has no edges, and $u$ is the only degree-4 neighbor of $v$. Otherwise, if $G[N(v)]$ contains an edge, then $\Phi(G[X]) = 6$, so by Claim \ref{2-chain-edges-leaving} and since $\Lambda = 2$, $G \setminus X$ must contain a triangle component. Let $X'$ consist of $X$ and the vertex set of the difficult components in $G \setminus X$. By Claim \ref{2-chain-edges-leaving} and since $\Phi(G[X]) \leq 7$, $G \setminus X$ cannot contain two 2-chains of triangles. If $G \setminus X$ contains two triangle components, then since $G$ contains no adjacent triangles, the vertices of each triangle component must be adjacent to every neighbor of $v$. Thus, $G[X']$ is as in Figure \ref{subfig:M=10-tri-tri}, so Claim \ref{main-inequality} with $X'$ playing the role of $X$ implies
    \begin{equation*}
        4 + 10(a - 1) + 16b < 0.
    \end{equation*}
    However, this contradicts \eqref{deg-4-deg-3-nbr-difficult-equation}. If $G \setminus X$ contains a triangle component and a 2-chain of triangles, note that as before, the vertices of the triangle component must be adjacent to every neighbor of $v$. Also note that $u$ cannot be adjacent to two vertices in the same triangle, or else $G$ contains adjacent triangles. Thus, $G[X']$ is as in Figure \ref{subfig:M=10-tri-2chain}, which has an independent set of size five, so Claim \ref{main-inequality} with $X'$ playing the role of $X$ implies
    \begin{equation*}
        5 + 13(a - 1) + 20b < 0.
    \end{equation*}
    However, this contradicts \eqref{deg-4-deg-3-nbr-2-difficult-equation}.\\

    \begin{figure}[H]
        \centering
        \subfloat[\centering $u$ has no neighbors in $N(v)$]{


              
            \includegraphics[]{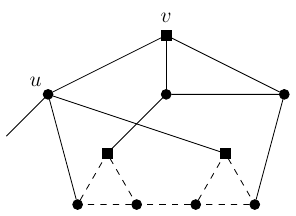}
            \label{subfig:M=9-u-not-in-tri}
        }
        \subfloat[\centering $u$ has one neighbor in $C$]{


              
            \includegraphics[]{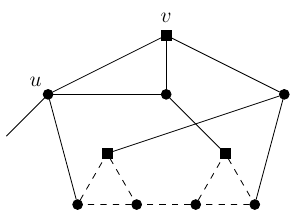}
            \label{subfig:M=9du=3}
            
        }
        \subfloat[\centering $u$ has two neighbors in $C$]{
              

              
            \includegraphics[]{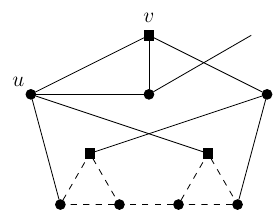}
            \label{subfig:M=9du=4}
        }
        
        \caption{The configurations in Case 2. Edges of difficult components are depicted with dashed lines, and any independent set in $G \setminus X'$ can be extended by adding the square vertices.}
        \label{fig:deg-3-deg-4-nbr-M=9}
    \end{figure}
    
    \textbf{Case 2.} $M = 9$ and $\Lambda = 1$. 
    
    In this case, $G[N(v)]$ has one edge and $u$ is the only degree-4 neighbor of $v$. As above, $G \setminus X$ has no triangle components, so $G \setminus X$ contains a 2-chain of triangles, $C$. Let $X'$ consist of $X \cup V(C)$. Note that adjacent neighbors of $v$ cannot both be adjacent to the same triangle, or else they belong to a 4-cycle adjacent to a triangle. Since $G$ does not contain adjacent triangles, a vertex in $N(v)$ cannot be adjacent to two vertices in the same triangle of $C$. Therefore if $u$ is not adjacent to a neighbor of $v$, then $G[X']$ must be as in Figure \ref{subfig:M=9-u-not-in-tri}. If $u$ is part of some triangle, then note that the vertex in $N(v) \setminus (N(u) \cup \{u\})$ must be adjacent to one vertex from each triangle in $C$. Moreover, $u$ is adjacent to either one or two vertices of $C$. In the first case $G[X']$ is as in Figure \ref{subfig:M=9du=3}, and in the second case, $G[X']$ is as in Figure \ref{subfig:M=9du=4}. Regardless of the specific configuration, Claim \ref{main-inequality} with $X'$ playing the role of $X$ implies
    \begin{equation*}
        3 + 10(a-1) + 16b < 0.
    \end{equation*}
    However, this contradicts \eqref{deg-4-deg-3-nbr-difficult-equation}.
\end{proof}

\begin{claim}\label{max-deg-3}
    The maximum degree of $G$ is at most three.
\end{claim}
\begin{proof}
    Suppose towards a contradiction that $G$ contains a degree-4 vertex $v$. By Claim \ref{deg-4-no-deg-3-nbrs}, every neighbor of $v$ has degree four. Let $X$ consist of $v$ and its neighbors, so $A = 1$ and $N = 5$. Because $v$ can be a part of at most two triangles, $M \geq 14$. Claim \ref{main-inequality} implies
    \begin{equation*}
        1 + 5(a - 1) + Mb < \Lambda b.
    \end{equation*}
    By \eqref{deg-4-deg-4-nbr-equation}, $M - \Lambda \leq 13$, so $\Lambda > 0$.
    Therefore $G \setminus X$ contains a difficult component $C$. Let $u \in V(C)$ such that $d_C(u) = 2$. By Claim \ref{min-deg-3}, $|N_G(u) \cap X| \geq 1$. By Claim \ref{deg-4-no-deg-3-nbrs}, and since every vertex in $X$ has degree four, $d_G(u) \geq 4$. Thus, $|N_G(u) \cap X| \geq 2$ and $\Phi(C) \geq 6$ since $C$ has at least three vertices $w$ with $d_C(w) = 2$. This implies that $v$ and $u$ share two neighbors, so $v$ cannot be part of a triangle or else $G$ contains a 4-cycle adjacent to a triangle. Therefore $M = 16$, so $\Lambda \geq 3$. Since $\Phi(C) \geq 6$ for every such difficult component $C$, $\Lambda \leq \lfloor \Phi(G[X])/6 \rfloor = 2$, a contradiction. 
\end{proof}

By Claim \ref{min-deg-3} and Claim \ref{max-deg-3}, $G$ is 3-regular. Also note since $G$ contains no triangles sharing an edge, $G$ has no subgraph isomorphic to $K_4$. Thus, by Brooks' Theorem \cite{brooks1941colouring}, $G$ is 3-colorable. By taking $X = V(G)$ and extending the independent set by the largest color class of any 3-coloring, Claim \ref{main-inequality} implies
\begin{equation*}
    \frac{1}{3}|V(G)| + (a - 1)|V(G)| + \frac{2}{3}|V(G)|b < 0.
\end{equation*}
However, this contradicts \eqref{3-reg-equation}, completing the proof of the theorem.\qed

\section{Proof of Theorem \ref{triangle-bordering-4-cycle-IR-thm}}\label{density-4-cycle-free-section}
In order to prove Theorem \ref{triangle-bordering-4-cycle-IR-thm}, we first need the following lemma.

\begin{lemma}\label{edge-bound}
    If $G$ is an $n$-vertex planar graph with no triangle adjacent to a 4-cycle and $n \geq 4$, then $|E(G)| \leq \frac{15}{7}(n - 2)$.
\end{lemma}
\begin{proof}
    Suppose $G$ is connected. Take an arbitrary planar embedding of $G$. Let $F(G)$ be the set of faces in the embedding. Let $\ell(f)$ be the length of each face $f \in F(G)$. Let $X$ be the set of faces that have length three. Let $Y$ be the set of faces that have length four. Note that $X$, $Y$ and $F(G)\setminus(X\cup{Y})$ partition $F(G)$, and all faces of length less than five are in $X \cup Y$, since $n \geq 4$. Thus, 
    \begin{equation}\label{"intermediate-inequality-1"}
        \begin{split}
        2|E(G)| &= \sum_{f\in{F(G)}}\ell(f) \\
        &= \sum_{f \in X}\ell(f) + \sum_{f \in Y}\ell(f) + \sum_{f \in F(G)\setminus{X\cup Y}}\ell(f) \\
        &\geq 3|X| + 4|Y| + 5(|F(G)| - |X| - |Y|) \\
        &= 5|F(G)| - 2|X| - |Y|.
        \end{split}
    \end{equation}
    
     We claim that 
     \begin{equation}\label{"lower-bound-edges-1"}
         \frac{2}{3}|E(G)| \geq 2|X| + |Y|.
     \end{equation}

    Observe that no edge on the boundary of a face of length three is also in a face of length four, or else a triangle and 4-cycle are adjacent. 
    Because $G$ is not a triangle and no two triangles are adjacent, no two length-3 faces share an edge. Therefore, we conclude that at least $3|X|$ edges of $G$ are on the boundary of a face in $X$. Moreover, every edge on the boundary of a length-4 face is on the boundary of at most two length-4 faces, so at least $2|Y|$ edges of $G$ are on the boundary of a face in $Y$. Thus, we have $|E(G)| \geq 3|X| + 2|Y|$, which implies \eqref{"lower-bound-edges-1"}. 
    
    Combining this with \eqref{"intermediate-inequality-1"} we get that $|F(G)| \leq 8|E(G)|/15$. Therefore, by Euler's Formula,
    \begin{equation*}
        |E(G)| = |V(G)| + |F(G)| - 2 \leq |V(G)| + \frac{8}{15}|E(G)| - 2.
    \end{equation*}
    After rearranging terms, we conclude that Lemma \ref{edge-bound} holds when $G$ is connected. 

    Suppose $G$ is not connected. Let $J$ be the set of components of $G$ with less than four vertices, and let $K$ be the set of components of $G$ with at least four vertices. Note that for all $j \in J$, $|E(j)| \le |V(j)|$. Therefore, if $|K| \ge 1$, Lemma \ref{edge-bound} holds by induction on $|J|$. Otherwise, $|K| = 0$. Note that  $n \le 15(n - 2)/7$ if $n \ge 4$, and in this case $|E(G)| \le |V(G)|$. Therefore, if the total number of vertices in the components in $J$ is at least four, Lemma \ref{edge-bound} holds. Thus, Lemma \ref{edge-bound} also holds when $G$ is not connected.
\end{proof}

Note that this bound matches the bound proved by Dowden \cite{dowden4cycleEdgeBound} for planar graphs without 4-cycles. Dowden also showed for infinitely many $n$ that this bound is tight even under the relaxed assumption that $G$ contains no 4-cycles.

\begin{proof}[Proof of Theorem \ref{triangle-bordering-4-cycle-IR-thm}]
   Let $a = 19/34$, and let $b = 3/34$. Note that $a$ and $b$ satisfy \eqref{deg-1-equation}--\eqref{3-reg-equation}. Let $G$ be a graph that is planar, has no $4$-cycle adjacent to a triangle, and satisfies $\lambda(G) = 0$. We can assume $|V(G)| \geq 4$, or else $G$ has an independent set on at least a third of its vertices. Theorem \ref{technical-thm} and Lemma \ref{edge-bound} imply 
   \begin{equation*}
   \alpha(G)\geq|V(G)|-(a|V(G)|+b|E(G)|) \geq \left(1 - a - \frac{15}{7}b\right)|V(G)| = \frac{1}{4-\frac{1}{30}}|V(G)|.
    \end{equation*}
    Therefore, $G$ has independence ratio at most $4 - 1/30$, as desired. Since every difficult component has an independent set on at least a third of its vertices, Theorem \ref{triangle-bordering-4-cycle-IR-thm} follows. 
\end{proof}

\section{Proof of Theorem \ref{Novosibirsk-IR-thm}}\label{density-novosibirsk-section}
In order to prove Theorem \ref{Novosibirsk-IR-thm}, we first need the following lemma.

\begin{lemma}\label{edge-bound-2}
    If $G$ is an $n$-vertex planar graph with no triangle adjacent to a cycle of length three or five and $n \geq 4$, then $|E(G)| \leq 2(n - 2)$.
\end{lemma}
\begin{proof}
    Suppose $G$ is connected. Take an arbitrary planar embedding of $G$. Let $F(G)$ be the set of faces in the embedding. Let $\ell(f)$ be the length of each $f \in F(G)$. Let $X$ be the set of faces that have length three or have length five but are not bounded by a 5-cycle. Let $Y$ be the set of faces that have length four, and let $Z$ be the set of faces bounded by a 5-cycle. Note that $X$, $Y$, $Z$, and $F(G)\setminus(X \cup Y \cup Z)$ partition $F(G)$, and all faces of length less than six are in $X \cup Y \cup Z$ since $n \geq 4$. Thus, 
    \begin{equation}\label{"intermediate inequality"}
        \begin{split}
        2|E(G)| &= \sum_{f\in{F(G)}}\ell(f) \\
        &= \sum_{f \in X}\ell(f) + \sum_{f \in Y}\ell(f) + \sum_{f \in Z}\ell(f) + \sum_{f \in F(G)\setminus{X \cup Y \cup Z}}\ell(f) \\
        &\geq 3|X| + 4|Y| + 5|Z| + 6(|F(G)| - |X| - |Y| - |Z|) \\
        &= 6|F(G)| - 3|X| - 2|Y| - |Z|.
        \end{split}
    \end{equation}
    
     We claim that 
     \begin{equation}\label{"lower bound on edges"}
         |E(G)| \geq 3|X| + 2|Y| + |Z|.
     \end{equation}

    Observe that no edge on the boundary of a face in $X$ is also on the boundary of a face in $Y \cup Z$, or else $G$ either contains a triangle adjacent to another triangle or a triangle adjacent to a 5-cycle. Therefore, it suffices to show that there are at least $3|X|$ edges on the boundary of a face in $X$ and at least $2|Y| + |Z|$ edges on the boundary of a face in $Y \cup Z$. Because $G$ is not a triangle and no two triangles are adjacent, no two faces in $X$ share an edge.  Thus, at least $3|X|$ edges of $G$ are on the boundary of a face in $X$. Moreover, every edge is on the boundary of at most two faces, so at least $2|Y| + |Z|$ edges of $G$ are on the boundary of a face in $Y \cup Z$. Thus, \eqref{"lower bound on edges"} holds.

    Combining \eqref{"lower bound on edges"} with \eqref{"intermediate inequality"} we have $|F(G)| \leq \frac{1}{2}|E(G)|$.
    Therefore, by Euler's Formula,
    \begin{equation*}
        |E(G)| = |V(G)| + |F(G)| - 2 \leq |V(G)| + \frac{1}{2}|E(G)| - 2.
    \end{equation*}
    After rearranging terms, we conclude that Lemma \ref{edge-bound-2} holds when $G$ is connected. As in the proof of Lemma \ref{edge-bound}, we can deduce that Lemma \ref{edge-bound-2} also holds when $G$ is not connected.
\end{proof}

The inequality stated by Lemma \ref{edge-bound-2} is tight for cylindrical grids with faces of length four.

\begin{proof}[Proof of Theorem \ref{Novosibirsk-IR-thm}]
    Let $a = 19/34$, and let $b = 3/34$. Note that $a$ and $b$ satisfy \eqref{deg-1-equation}--\eqref{3-reg-equation}. Let $G$ be a graph that is planar, has no triangle adjacent to a triangle or a 5-cycle, and satisfies $\lambda(G) = 0$. Clearly we can assume $|V(G)| \geq 4$, or else $G$ has an independent set on at least third of its vertices. Theorem \ref{technical-thm} and Lemma \ref{edge-bound-2} implies 
    \begin{equation*}
            \alpha(G) \geq |V(G)|-(a|V(G)|+b|E(G)|) \geq \bigl(1 - a - 2b\bigr)|V(G)| = \frac{1}{4-\frac{2}{9}}|V(G)|.
    \end{equation*} 
    Therefore, $G$ has an independence ratio of at most $4 - 2/9$, as desired. Since every difficult component has an independent set on at least a third of its vertices, Theorem \ref{Novosibirsk-IR-thm} follows. 
\end{proof}

\section{Concluding Remarks}\label{conclusion-section}
As illustrated by Figure \ref{fig:inequalitiesGraphed}, every $a,b \geq 0$ satisfying inequalities \eqref{deg-1-equation}, \eqref{3-chain-equation}, and \eqref{deg-5-equation} also satisfies inequalities \eqref{deg-1-equation}--\eqref{3-reg-equation}. Therefore, as mentioned in Section \ref{introduction-section}, Theorem \ref{technical-thm} actually holds with inequalities  \eqref{deg-2-equation} and \eqref{deg-4-deg-3-nbr-equation}--\eqref{3-reg-equation} omitted. However, we chose to include inequalities \eqref{deg-2-equation} and \eqref{deg-4-deg-3-nbr-equation}--\eqref{3-reg-equation} to improve the readability of the proof of Theorem \ref{technical-thm}.

Figure 4 also explains the choice of $(a, b) = (19/34, 3/34)$ in the proofs of Theorems \ref{Novosibirsk-IR-thm} and \ref{triangle-bordering-4-cycle-IR-thm}. To obtain the best possible bound on the independence ratio in Theorems \ref{Novosibirsk-IR-thm} and \ref{triangle-bordering-4-cycle-IR-thm} using Theorem \ref{technical-thm}, we choose $a,b \geq0$ satisfying  \eqref{deg-1-equation}--\eqref{3-reg-equation} to minimize $a + 2b$ and $a +15b/7$, respectively. Since  \eqref{deg-1-equation}--\eqref{3-reg-equation} are all linear inequalities, by the Fundamental Theorem of Linear Programming, the minimizers will occur at the corners of the polygonal region bounded by these constraints. The corners are (0, 1), (5/13, 3/13), (19/34, 3/34), and (1, 0), and (19/34, 3/34) is the minimizer in both cases.
\begin{figure}[H]
    \centering
    \includegraphics[]{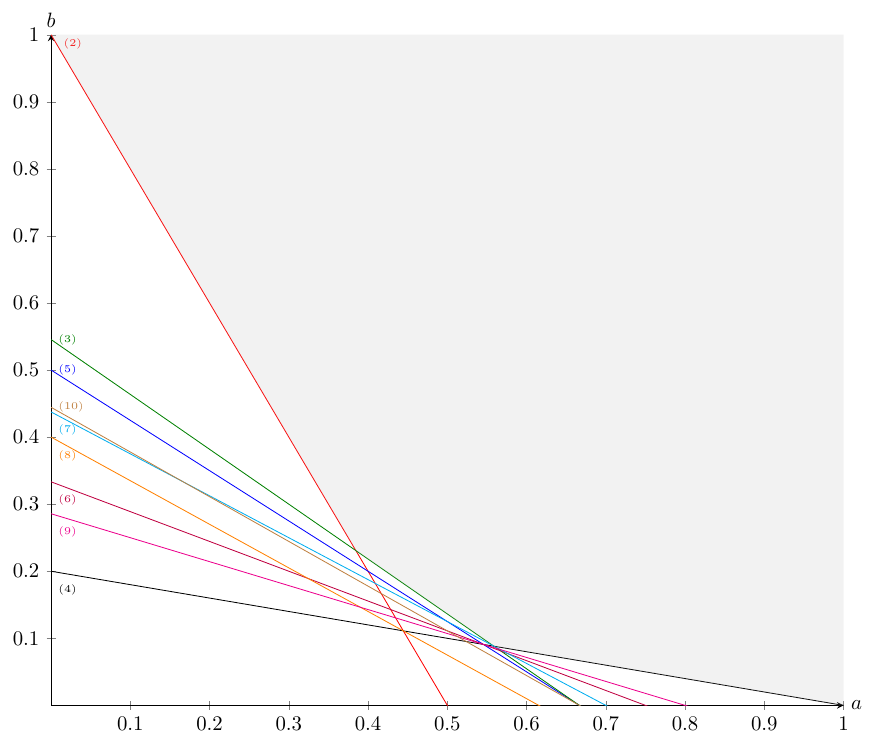}
    \caption{Plot of equations \eqref{deg-1-equation}--\eqref{3-reg-equation}}
    \label{fig:inequalitiesGraphed}
\end{figure}

\bibliographystyle{amsabbrv}
\bibliography{refs}
\end{document}